\documentclass[11pt]{amsart}

\usepackage{color}
\usepackage{amsmath}
\usepackage{amsthm}
\usepackage{amssymb}
\usepackage{graphicx}
\usepackage{enumerate}
\usepackage{amsfonts}
\usepackage{mathrsfs}
\usepackage{parskip}
\usepackage{mathdots}
\usepackage{float}

\usepackage{tikz}
\usetikzlibrary{calc}

\numberwithin{equation}{section}
\theoremstyle{plain}
\newtheorem{Proposition}[equation]{Proposition}
\newtheorem{Corollary}[equation]{Corollary}
\newtheorem*{Corollary*}{Corollary}
\newtheorem{Theorem}[equation]{Theorem}
\newtheorem*{Theorem*}{Theorem}
\newtheorem{Lemma}[equation]{Lemma}
\theoremstyle{definition}

\newtheorem{Example}[equation]{Example}
\newtheorem{Remark}[equation]{Remark}

\def\C{\mathbb{C}}
\def\R{\mathbb{R}}
\def\D{\mathbb{D}}
\def\T{\mathbb{T}}
\def\N{\mathbb{N}}

\def\phi{\varphi}
\def\leq{\leqslant}
\def\geq{\geqslant}
\def\subset{\subseteq}

\newcommand{\wt}[1]{\widetilde{#1}}

\title{The range and valence of a real Smirnov function}

\author[Ferguson]{Timothy Ferguson}
	\address{Department of Mathematics, University of Alabama, Tuscaloosa, AL}
	\email{tjferguson1@ua.edu}

\author[Ross]{William T. Ross}
	\address{Department of Mathematics and Computer Science, University of Richmond, Richmond, VA 23173, USA}
	\email{wross@richmond.edu}



\begin{document}

\begin{abstract}
  We give a complete description of the possible
  ranges of real Smirnov functions (quotients of two bounded analytic
  functions on the open unit disk where the denominator is outer and
  such that the radial boundary values are real almost everywhere on
  the unit circle). Our techniques use the theory of
  unbounded symmetric Toeplitz operators, some general theory of
  unbounded symmetric operators, classical Hardy spaces, and an
  application of the uniformization theorem. In addition, we
  completely characterize the possible valences for these real Smirnov
  functions when the valence is finite. To do so we construct
  Riemann surfaces we call disk trees by welding together copies of
  the unit disk and its complement in the Riemann sphere.
  We also make use of certain trees we call valence trees that
  mirror the structure of disk trees. 
\end{abstract}

\maketitle

\section{Introduction}
This paper explores the range and valence of real Smirnov functions. Smirnov functions, a well studied class of functions \cite{Duren}, are analytic functions on the open unit disk $\D$ which can be written as the quotient of two bounded analytic functions where the denominator is an outer function. Real Smirnov functions, studied in \cite{GMR, MR2021044, Helson, Helson2, MR1889082}, are those Smirnov functions which have real boundary values almost everywhere.  In a nutshell, we will characterize all possible ranges of such functions and all possible finite valences on their range. The two main theorems of this paper (terminology, motivation, and plenty of examples  to be reviewed below) are the following: 

\begin{Theorem*}
If $\phi$ is a non-constant real Smirnov function, then $\phi(\D)$ is either 
$$\phi(\D) =  \C_{+} \setminus F \; \; \mbox{or} \; \; \phi(\D) = \C_{-} \setminus G \; \; \mbox{or} \; \; \phi(\D) = \C \setminus (F \cup G \cup E),$$
 where $E \subsetneq \R$ and closed, $F \subset \C_{+}$ is relatively closed and has  
capacity zero, and $G \subset \C_{-}$ is relatively closed and has capacity zero. 
Moreover, given any closed $E \subsetneq \R$, any relatively closed $F \subset \C_{+}$ of  capacity zero, and any relatively closed 
$G \subset \C_{-}$ of capacity zero,  there are real Smirnov functions
with ranges  $\C_{+} \setminus F$, $\C_{-} \setminus G$, 
and $\C \setminus (E \cup F \cup G)$.
\end{Theorem*}

\begin{Theorem*}
The valence of every real Smirnov function with finite valence is given 
by the valence of a plane valence tree, and any valence arising from a 
plane valence tree is the valence of a real Smirnov function.
\end{Theorem*}

The inspiration for this paper, and what informs our results, comes from the study of unbounded Toeplitz operators on the Hardy space $H^2$ of the open unit disk $\D$ (see \cite{MR3004956, MR2418122} and below). Here, for a general analytic function $\phi$ on $\D$, one can define the Toeplitz operator 
$$T_{\phi}: \mathcal{D}(T_{\phi}) \to H^2, \quad T_{\phi} f = \phi f,$$
where $\mathcal{D}(T_{\phi})$, the domain of $T_{\phi}$, is defined by 
$$\mathcal{D}(T_{\phi}) = \{f \in H^2: \phi f \in H^2\}.$$ Sarason \cite{MR2418122} showed that $\mathcal{D}(T_{\phi}) \not = \{0\}$ if and only if 
\begin{equation}\label{99sdhsfh}
\phi = \frac{b}{a},
\end{equation} where $b$ and $a$ are bounded analytic functions on $\D$ and $a$ has no zeros. Such $\phi$ comprise the well-known {\em Nevanlinna class} $N$ \cite{Duren}. It can also be arranged so that 
$$a(0) > 0 \; \; \mbox{and} \; \; |a|^2 + |b|^2 = 1$$ almost everywhere on the unit circle $\T$. With these normalizing conditions, this representation is unique. In the same paper, Sarason also showed that $\mathcal{D}(T_{\phi})$ is dense in $H^2$ if and only $a$ in \eqref{99sdhsfh} is an outer function. These $\phi$ comprise the {\em Smirnov class} $N^{+}$. Observe from \cite[Ch.~2]{Duren} that 
\begin{equation}\label{bcvcdf}
\bigcup_{p > 0} H^p \subset N^{+},
\end{equation}
 where $H^p$, the {\em Hardy classes}, are the analytic functions $f$ on $\D$ for which the $p$-integral means
\begin{equation}\label{Mp}
M_{p}(r, f) := \left(\int_{0}^{2 \pi} |f(r e^{i t})|^p \frac{d t}{2 \pi}\right)^{1/p}
\end{equation} are uniformly bounded for $r \in [0, 1)$. 

Classical theorems of Fatou and Riesz \cite[Ch.1, 2]{Duren} say that for each $\phi \in N^{+}$ the radial limit 
\begin{equation}\label{radiallimit}
\phi(e^{i t}) := \lim_{r \to 1^{-}} \phi(r e^{i t})
\end{equation}
 exists (and is non-zero) for almost every $t \in [0, 2 \pi]$. 
We say $\phi \in N^{+}$ belongs to the {\em real Smirnov class} $N^{+}_{\R}$ if
$$\phi(e^{i t}) \in \R$$ 
for almost every $t$ (see some examples below). 
These real Smirnov functions have been studied in \cite{GMR, MR2021044, Helson, Helson2, MR1889082} and a full characterization of them was given by Helson \cite{Helson, Helson2} as 
\begin{equation}\label{sdhfjsd;gfgee2}
\phi \in N^{+}_{\R} \iff \phi = i \frac{u + v}{u - v},
\end{equation}
 where $u$ and $v$ are inner functions and $u - v$ is an outer function.

When $\phi \in N^{+}_{\R}$ and 
$$\langle f, g \rangle = \int_{0}^{2 \pi} f(e^{i t}) \overline{g(e^{i t})} \frac{d t}{2 \pi}$$
denotes the usual inner product  
on $H^2$ (considered in the usual way, via radial limit functions, as a closed subspace of $L^2$), one can use
the fact that  $\phi(e^{i t})$ is real for almost every $t$ to see that $$\langle T_{\phi} f, g\rangle = \langle f, T_{\phi} g\rangle, \quad f, g \in \mathcal{D}(T_{\phi}).$$ In other words, $T_{\phi}$ is a densely defined symmetric operator on $H^2$.

Standard results from the theory of unbounded symmetric operators  \cite[Vol. II, Ch.~VII]{A-G} show that when $\phi \in N^{+}_{\R}$ and $\lambda \not \in \R$, the densely defined operator $T_{\phi} - \lambda I$ has closed range and the deficiency numbers
\begin{equation}\label{668wyeuhfjw}
d_{\phi}(\lambda) := \operatorname{dim}(\operatorname{Rng}(T_{\phi} - \lambda I))^{\perp},
\end{equation} 
where $\operatorname{Rng}$ denotes the range of an operator,
are constant on each of the half planes 
$$\C_{+} = \{z: \Im z > 0\}, \; \; \C_{-} = \{z: \Im z < 0\}.$$ Moreover, given a pair $(m, n)$, where $m, n \in \N_{0} \cup \{\infty\}$, there is a $\phi \in N^{+}_{\R}$ with 
$$(d_{\phi}(i), d_{\phi}(-i)) = (m, n).$$ It is also the case that both deficiency numbers are finite if and only if $\phi \in N^{+}_{\R}$ is a rational function.

To get to our discussion of the range of $\phi$, the focus of this paper, we unpack this a bit further as was done in \cite{Helson}. Observe that $\operatorname{Rng}(T_{\phi} - \lambda I)$ is not only a closed subspace of $H^2$ but it is also invariant under the shift operator $Sf = z f$ on $H^2$ and thus, by Beurling's theorem \cite[Ch.~7]{Duren},
$$\operatorname{Rng}(T_{\phi} - \lambda I) = \Theta H^2$$
for some inner function $\Theta$.
Let $\Theta_{\lambda}$ denote the inner factor of $\phi - \lambda$. 
All functions in
$\Theta H^2$ have $\Theta_{\lambda}$ as a divisor.
Moreover, since we can write $\phi - \lambda = b/a$ where $b$ and $a$ are in
$H^\infty$ and $a$ is outer, it follows that the inner factor of
$(\phi - \lambda)$ is the inner factor of $b$, and thus the
inner factor of $(\phi - \lambda)a$ is precisely $\Theta_{\lambda}$.
Thus $\Theta = \Theta_\lambda$. 
This means that 
$$(\operatorname{Rng}(T_{\phi} - \lambda I))^{\perp} = (\Theta_{\lambda} H^2)^{\perp}$$ which is a model space \cite{MR1761913, MR0270196, MR3526203, MR1895624}, a typical invariant subspace for the backward shift operator $S^{*}$. Moreover, the model space $(\Theta_{\lambda} H^2)^{\perp}$ has finite dimension $n$ if and only if $\Theta_{\lambda}$ is a finite Blaschke product of degree $n$. Since 
$$\ker(T^{*}_{\phi}  - \overline{\lambda} I) = (\operatorname{Rng}(T_{\phi} - \lambda I))^{\perp} = (\Theta_{\lambda} H^2)^{\perp},$$
and for each $w \in \D$,
$$T^{*}_{\phi} k_{w} = \overline{\phi(w)} k_{w}, \quad k_{w}(z) = \frac{1}{1 - \overline{w} z},$$
we see that 
$$\bigvee \{k_{w}: \phi(w) = \lambda\} \subset \ker(T^{*}_{\phi} - \overline{\lambda} I) = (\Theta_{\lambda} H^2)^{\perp}.$$
In the above, $\bigvee$ denotes the closed linear span. 
Hence, using the fact that $\Theta_{\lambda}$ is the inner factor for $\phi - \lambda$, we see that  for $\lambda \in \C \setminus \R$,
$$\lambda \in \phi(\D) \iff \Theta_{\lambda}(w) = 0 \; \; \mbox{for some $w \in \D$}.$$
 Furthermore, the valence 
 \begin{equation}\label{bbbbbbx}
v_{\phi}(\lambda) := \operatorname{card}\{w \in \D: \phi(w) = \lambda\}
\end{equation} will be the degree of the Blaschke factor of $\Theta_{\lambda}$.
For example, if the inner factor of $\Theta_{\lambda}$ is either a unimodular constant or a singular inner function (which will have no zeros in $\D$), then $\lambda \not \in \phi(\D)$, i.e., $v_{\phi}(\lambda) = 0$.  On the other hand, if $\Theta_{\lambda}$ has is an infinite Blaschke factor, then $v_{\phi}(\lambda) = \infty$. Note that 
$$v_{\phi}(\lambda) \leq d_{\phi}(\lambda), \quad \lambda \not \in \R,$$ but equality does not always hold. For example, $\Theta_{\lambda}$ might be the product of a finite Blaschke product of degree $n$ and a singular inner function. In this case $v_{\phi}(\lambda) = n$ while $d_{\phi}(\lambda) = \infty$.

Thus, characterizing
the range of $\phi$ will involve a discussion of the $\lambda \in \C \setminus \R$ such that $\phi - \lambda$  has a non-trivial Blaschke factor.
Rudin \cite{MR0235151}, generalizing a classical theorem of Frostman \cite[p.~37]{C-L}, showed that for nearly all $\lambda \in \C \setminus \R$, the inner factor of $\phi - \lambda$ is a Blaschke product. Here ``nearly all'' means that this property holds with the possible exceptional set of logarithmic capacity (capacity for short) zero. See \cite{Fisher, MR1334766} for basic facts about logarithmic capacity and see \cite{MR779463, GRM, MR2986324} for more on Blaschke products. 

In the above, we are allowing a unimodular constant to count as a Blaschke factor (of order zero). In this degenerate case we see that $\Theta_{\lambda} \equiv \xi$ for some $\xi \in \T$ and so 
$$(\operatorname{Rng}(T_{\phi} - \lambda I))^{\perp} = (\xi H^2)^{\perp} = \{0\}.$$ Thus if $\Theta_{\lambda}$ is a unimodular constant function for one $\lambda \in \C_{+}$ (or one $\lambda \in \C_{-})$ then, since $d_{\phi}$ is constant on each of $\C_{+}$ or $\C_{-}$, it follows that $\Theta_{\lambda}$ is a constant unimodular function for all $\lambda \in \C_{+}$ (or all $\lambda \in \C_{-}$). If $\Theta_{\lambda}$ is a constant unimodular function for one $\lambda \in \C_{+}$, then $\phi(\D) \cap \C_{+}  = \varnothing$ (similarly for some $\lambda \in \C_{-}$ and hence $\phi(\D) \cap \C_{-} = \varnothing$). Hence, for example, if $\phi(\D) \cap \C_{+}$ omits an open disk about $\lambda \in \C_{+}$, then $\phi - \lambda$ is an outer function, i.e., $\Theta_{\lambda}$ is a constant unimodular function. In this case the above discussion implies that $\phi(\D) \cap \C_{+} = \varnothing$. 

Thus, using the above analysis, along with the fact that $\phi(\D)$ is an open connected subset of $\C$ (open mapping theorem), we have the following possibilities for the range of $\phi \in N^{+}_{\R}$:
\begin{equation}\label{789ruwoiepdws}
\phi(\D) =  \C_{+} \setminus F \; \; \mbox{or} \; \; \phi(\D) = \C_{-} \setminus G \; \; \mbox{or} \; \; \phi(\D) = \C \setminus (F \cup G \cup E),
\end{equation}
where $F \subset \C_{+}$ and $G \subset \C_{-}$ are relatively closed subsets of capacity zero and $E \subsetneq \R$ is closed. The question we ask and answer in this paper is whether or not we can actually obtain all of these possibilities. 
In addition, we also discuss the valance on these ranges. 

To give the reader a feel for where we are heading, let us consider a few simple examples of ranges of $\phi \in N_{\R}^{+}$. 

\begin{Example}
If 
$$\phi_1(z) = i \frac{1 + z}{1 - z},$$ 
then $\phi_1 \in N^{+}$ (since it is the 
quotient of two bounded analytic functions 
and the denominator $1 - z$ is outer) and 
$$\phi_1(e^{i t}) = - \cot(t/2) \in \R, $$ 
which says that $\phi_1 \in N_{\R}^{+}$. 
Furthermore, $\phi_1(\D) = \C_{+}$. 
In a similar way, we see that if 
$$\phi_2(z)  = -i \frac{1 + z}{1 - z},$$ then $\phi_2 \in N_{\R}^{+}$ and $\phi_2(\D) = \C_{-}$. 

If $\theta$ is the singular inner function
$$\theta(z) = \exp\left(\frac{1 + z}{1 - z}\right),$$
then $\theta(\D) = \D \setminus \{0\}$ and thus if $\psi_{1} := \phi_1 \circ \theta$, then $\phi_1 \in N^{+}_{\R}$ and 
$\psi_{1}(\D) = \C_{+} \setminus \{i\}$. Observe that the singleton $\{i\}$ has capacity zero \cite[p.~140]{MR1334766}. 

Given a relatively closed subset $W \subset \D$ of capacity zero, there is an inner function $\sigma$ such that $\sigma(\D) = \D \setminus W$ \cite{C-L}. Then $\psi_2 = \phi_1 \circ \sigma \in N^{+}_{\R}$ and $\psi_{2}(\D) = \C_{+} \setminus F$, where $F = \psi_1(W)$ has capacity zero. 
\end{Example}

\begin{Example}
If 
$$\phi_3(z) = \left(\frac{1 + z}{1 - z}\right)^4$$ then 
$$\phi_{3}(e^{i t}) = \cot^{4}(t/2) \in \R$$ and, 
since 
$$z \mapsto \frac{1 + z}{1 - z}$$
maps $\D$ onto the right-half plane $\{z: \Re z > 0\}$, then $\phi_{3}(\D) = \C \setminus \{0\}$. 
\end{Example}

\begin{Example}
If 
\begin{equation*}%
\phi_4(z) = \frac{z}{(1 - z)^2},
\end{equation*}
the well-known Koebe function, then 
$$\phi_4(e^{i t}) = -\frac{1}{2} \frac{1}{1 - \cos t} \in \R$$ and 
$\phi_4(\D)$ is the single slit domain $\C \setminus (-\infty, -\tfrac{1}{4}]$.
\end{Example}

 \begin{Example}
 If 
$$\phi_5(z) = \frac{i z}{1 - z^2},$$
then 
$$\phi_5(e^{i t}) = - \frac{1}{2} \csc t \in \R$$ and $\phi_5(\D)$ turns out to be the double slit domain 
$$\C  \setminus ((-\infty, -\tfrac{1}{2}] \cup [\tfrac{1}{2}, \infty)).$$
\end{Example}

Of course one can compose any of the functions $\phi_j$ from these examples with interesting inner functions, like was done in the first example, to obtain ranges taking the form $\C_{+} \setminus F$, $\C_{-} \setminus G$, and $\C \setminus (F \cup G \cup E)$ for relatively closed sets $F$ and $G$ of capacity zero and a closed set $E \subsetneq \R$.  Can we obtain all of these possibilities as ranges for given $E, F, G$? 

\section{The main range result}

Our main result about the range of $\phi \in N_{\R}^{+}$ is the following: 

\begin{Theorem}\label{MT}
If $\phi \in N^{+}_{\R}$  and non-constant, then $\phi(\D)$ is either 
$$\phi(\D) =  \C_{+} \setminus F \; \; \mbox{or} \; \; \phi(\D) = \C_{-} \setminus G \; \; \mbox{or} \; \; \phi(\D) = \C \setminus (F \cup G \cup E),$$
 where $E \subsetneq \R$ and closed, $F \subset \C_{+}$ is relatively closed and has  
capacity zero, and $G \subset \C_{-}$ is relatively closed and has capacity zero. 
Moreover, given any closed $E \subsetneq \R$, any relatively closed $F \subset \C_{+}$ of  capacity zero, and any relatively closed 
$G \subset \C_{-}$ of capacity zero,  there are functions in $ N^{+}_{\R}$ 
with ranges  $\C_{+} \setminus F$, $\C_{-} \setminus G$, 
and $\C \setminus (E \cup F \cup G)$.
\end{Theorem}

Our  proof needs a variation of a result from \cite[p.~119]{Fisher}. 

\begin{Lemma}\label{99w8w8w8w8w+}
Suppose $f$ is a non-constant function belonging to $H^p$ for some $p \in (0, \infty)$ and $\mathcal{E} \subset \T$ of positive Lebesgue measure for which 
$$\lim_{r \to 1^{-}} f(r \xi) =: f(\xi)$$ exists for
each $\xi \in \mathcal{E}$ . If 
$E = \{f(\xi): \xi \in \mathcal{E}\}$, then $E$ has positive capacity. 
\end{Lemma}

\begin{proof}
  We follow the original proof from \cite[p.~119]{Fisher} with a slight variation.
  First, note that since the almost everywhere defined boundary $\xi \mapsto f(\xi)$ function on $\T$ is a limit of
  continuous functions (the dilations $f_r(\xi) := f(r \xi)$) on $\T$,  
  by Egorov's theorem there is a set of positive measure
  that is a subset of $\mathcal{E}$ on which the boundary function
  $f$ is continuous. By the 
  inner regularity of Lebesgue measure, this set has a compact
  subset of positive measure.  Without loss of generality we may
  take $\mathcal{E}$ to be this set. Then $f(\mathcal{E})$ is compact,
  and
  $|f| \leq K$ on $\mathcal{E}$ for some $K > 0$.
    Suppose towards a contradiction that $E$ has zero logarithmic capacity. Then
    by Evan's theorem 
    \cite[p.~33]{Fisher}, for some probability measure $\sigma$ on $\T$, the logarithmic potential 
$$p(z) = - \int_{E} \log |z - \zeta| d\sigma(\zeta)$$
satisfies 
$$\lim_{z \to \xi} p(z) = +\infty, \quad \xi \in E.$$
The function $u = p(f)$ is harmonic on $\D$ and satisfies the condition 
$$\lim_{r \to 1^{-}} u(r \xi) = + \infty, \quad \xi \in \mathcal{E}.$$
Let $v$ be the harmonic conjugate of $u$ on $\D$ and define 
$F = e^{-u - i v}$.
Assuming that $F \in H^p$ (a fact we will prove momentarily), we see that 
$$\lim_{r \to 1^{-}} |F(r \xi)| = 0, \quad \xi \in \mathcal{E}.$$
But since $\mathcal{E}$ has positive Lebesgue measure, this would mean that $f \equiv 0$ \cite[p.~17]{Duren}, a contradiction. 

To finish,  we  now show that $F \in H^p$. Observe that a use of Jensen's inequality \cite[p.~43]{MR1334766}  and the integral means definition of $H^p$ from \eqref{Mp}
shows that 
\begin{align*}
\int_{0}^{2 \pi} |F(r e^{i t})|^p \frac{d t}{2 \pi} & = \int_{0}^{2 \pi} \exp\left(\int_{E} \log|f(r e^{i t}) - \zeta|^p d \sigma(\zeta)\right)\frac{d t}{2 \pi}\\
& \leqslant \int_{0}^{2 \pi} \left(\int_{E} |f(r e^{i t}) - \zeta|^p d \sigma(\zeta)\right) \frac{d t}{2 \pi}\\
& = \int_{E}\left( \int_{0}^{2 \pi}  |f(r e^{i t}) - \zeta|^p \frac{dt}{2 \pi}\right) d \sigma(\zeta)\\
& \leqslant 2^p (\sup_{0 < r < 1} M(r, f)^{p} + K^p).
\end{align*}
Since the last quantity above is independent of $r \in (0, 1)$, this shows that $F \in H^p$ and thus completes the proof. 
\end{proof}

\begin{proof}[Proof of Theorem \ref{MT}]

  The first part of the theorem follows from the discussion preceding \eqref{789ruwoiepdws}. For the second part (obtaining all possible types of ranges), we proceed as follows.
  First consider the case where the range contains points in both the
  upper and lower half planes. 
Since $E$ is a proper closed subset of $\R$, it follows that  $\R \setminus E$ has at least one component. If it has {\em exactly one} component, then $E$ must be one of the following four options: 
\begin{equation}\label{s8dufds}
\varnothing, \quad (-\infty, c], \quad [c, \infty), \quad (-\infty, a] \cup [b, \infty) , \quad (a < b).
\end{equation}

We will first deal with the case where $\R \setminus E$
has {\em at least two} components.  
By means of a real translation of our function at the end, we can assume $0 \in E$ and, for some $a <0$ and $0 < b < c$, that 
\begin{equation}\label{3094ouitrgjlfksd}
E \subset (- \infty, a] \cup [0, b] \cup [c, \infty).
\end{equation}
Define 
$$E_{+} = E \cap [0, \infty), \; \; E_{-} = E \cap (-\infty, 0)$$ and the open set  
\begin{equation}\label{Omega}
\Omega = \{\Re z > 0\} \setminus (\wt{E}_1 \cup \wt{E}_2 \cup \wt{E}_3 \cup \wt{F} \cup \wt{G}),
\end{equation}
where 
$$\wt{E}_1 = \{x^{\frac{1}{4}}: x \in E_{+}\},$$
$$\wt{E}_2 =  e^{i \frac{\pi}{4}} \{(-x)^{\frac{1}{4}}: x \in E_{-}\},$$
$$\wt{E}_3 = e^{-i \frac{\pi}{4}} \{(-x)^{\frac{1}{4}}: x \in E_{-}\},$$
and $\wt{F}$ and $\wt{G}$ are the intersections of the right half plane with the
 images of $F$ and $G$ respectively under the multivalued map 
$z \mapsto z^{1/4}$.  
\begin{figure}
\begin{tikzpicture}[scale=2.0]
\draw (-2,0) node{};
\draw[thick] (0,0)--(1,0) node[anchor=west]{$b^{1/4}$};
\draw (.5,0) node[anchor=south]{$\widetilde{E}_1$};
\draw[thick,->] (2,0) node[anchor=east]{$c^{1/4}$} -- (4,0) 
    node[anchor=west]{$\widetilde{E}_1$};
\draw[thick,->] (1,1) node[anchor=north west]{$|a|^{1/4}e^{\pi i/4}$} -- (2,2) node[anchor = west]{$\widetilde{E}_2$};
\draw[thick,->] (1,-1) node[anchor= south west]{$|a|^{1/4}e^{-\pi i/4}$} -- (2,-2)node[anchor = west]{$\widetilde{E}_3$};
\draw[thick,dashed] (1,.25)--(1.25,.3125) node[above]{$\widetilde{F}$};
\draw[thick,dashed] (.25,-1)--(.3125,-1.25)node[right]{$\widetilde{F}$};
\draw[thick] (0,2) -- (0,-2);
\draw[dashed] (.5,-.25) --( .75,-.25)node[below]{$\widetilde{G}$};
\draw[dashed] (.25,.5) --( .25,.75)node[above]{$\widetilde{G}$};

\draw[fill=black] (1,0) circle(.05cm);
\draw[fill=black] (2,0) circle(.05cm);
\draw[fill=black] (1,1) circle(.05cm);
\draw[fill=black] (1,-1) circle(.05cm);

\draw (2.5,1.5) node{\LARGE $\Omega$};

\end{tikzpicture}
\caption{The region $\Omega$ when $E = (- \infty, a] \cup [0, b] \cup [c, \infty).$}\label{figa}
\end{figure}
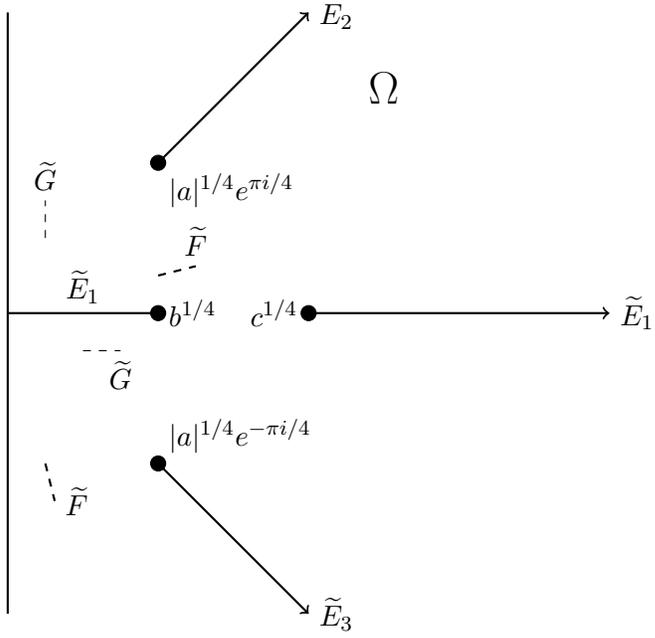
See Figure \ref{figa} for an illustration of $\Omega$ when 
$$E = (- \infty, a] \cup [0, b] \cup [c, \infty).$$
Since $\Omega$ is contained in 
$$\{\Re z > 0\} \setminus \left([0, b^{\frac{1}{4}}] \cup [c^{\frac{1}{4}}, \infty) \cup e^{i \pi/4} [|a|^{\frac{1}{4}}, \infty) \cup e^{- i\pi/4} [|a|^{\frac{1}{4}}, \infty) \right)$$ 
and this last set is connected, one concludes, also using the containment in \eqref{3094ouitrgjlfksd} along with the fact that extracting a set of capacity zero does not disconnect a domain \cite[p.~68]{MR1334766},  that $\Omega$ is connected. 
By \cite[p.~125]{MR1344449} there exists an analytic covering map for $\Omega$, such that $\psi(\D) = \Omega$. Furthermore, since $\psi(\D)$ is contained in a half-plane, then $\psi$ belongs to $H^{p}$ for all $p \in (0, 1)$ \cite[p.~109]{Garnett}. Thus, by \eqref{bcvcdf}, $\psi \in N^{+}$. 

The map $\psi$ is a covering map from $\D$ to $\Omega$, which 
means that each point of $\Omega$  is contained in an open neighborhood $U$ such that 
$\psi^{-1}(U)$ consists of disjoint open sets each of which is homeomorphic to $U$ 
under $\psi$. 

We now claim that if the radial limit
\begin{equation}\label{oos9d9}
\lim_{r \to 1^{-}} \psi(r e^{i t})
\end{equation}
 exists, which it will for almost every $t$ (see \eqref{radiallimit}), then this value must belong to $\partial \Omega$. 
Indeed, if this were not the case, then, for some particular $t$,  the limit would be equal to some 
$w \in \Omega$.  Now choose an open neighborhood $U$ of $w$ such that 
$\psi^{-1}(U)$ consists of disjoint sets each of which is homeomorphic to $U$ 
under $\psi$.  Let $W$ be an open neighborhood of $w$ contained in 
$U$ and such that 
$\overline{W} \cap U$ is compact.  
Then 
$$\psi^{-1}(W) = \bigcup_{a \in A} V_a,$$ where 
$A$ is some index set and the $V_a$ are pairwise disjoint open sets that are each
homeomorphic to $W$ under $\psi$.  
Thus each $V_a$ has compact closure in $\psi^{-1}(U)$.  
For some $b \in [0,1)$, the curve 
$$r \to \psi([r e^{i t},e^{i t})), \quad r \in [b, 1),$$
must lie entirely in $W$ since $\psi$ has radial limit of $w$ 
at $e^{i t}$. But since the $V_a$ are 
disjoint open sets, this means that the ray 
$$[r e^{i t},e^{i t}), \quad r \in [b, 1),$$
must lie entirely in one of the $V_a$.  But this is impossible because 
each of the $V_a$ has compact closure in $\psi^{-1}(U)$ but 
$e^{i t} \not \in \psi^{-1}(U)$.

The set $\wt{F}$ is the union of two images of $F$ under maps that are
Lipschitz in any annulus centered at the origin.  Since $F$ has
capacity $0$, each of the images has capacity $0$
\cite[p.~137]{MR1334766}, and their union $\wt{F}$ also has capacity
zero \cite[p.~57]{MR1334766}. The same applies to $\wt{G}$, and thus
$\wt{F} \cup \wt{G}$ has capacity zero.
By Lemma \ref{99w8w8w8w8w+}, the values in those sets cannot be
boundary values of $\psi$, except possibly on a set of measure $0$.

Setting $\phi = \psi^{4}$ we see that $\psi \in H^{p}$ 
for all $p \in (0, \tfrac{1}{4})$ and thus, again from \eqref{bcvcdf}, $\psi \in N^{+}$. Moreover, since 
\[
\lim_{r \to 1^{-}} \psi(r e^{i t}) \in \partial \Omega \setminus (\wt{F} \cup \wt{G})
\]
for almost every $\theta$, we see that 
$$\lim_{r \to 1^{-}} \phi(r e^{i t}) \in \R$$ for almost every $t$. Thus $\phi \in N_{\R}^{+}$. 
The construction of $\Omega$ from \eqref{Omega}, and the fact that $0 \in E$, will show that 
\[\phi(\D) = \{z^4: z \in \Omega\} = \C \setminus (E \cup F \cup G).\]

For the cases where the desired range is $\C_{+} \setminus F$ or $\C_{-} \setminus G$, and for the case where 
 $E = (-\infty, c]$ or $E = [c, \infty)$ or $E = (-\infty, a] \cup [b, \infty)$ and the desired range is 
$\C \setminus (E \cup F \cup G)$, we let $\phi$ be the covering map from $\D$ onto the desired range.  The proof that this map has the required properties is similar to the proof of the 
first case above.  We use the fact that any mapping from $\D$ into a domain with at least one slit that 
goes to $\infty$ is in $H^p$ for all $p \in (0,\tfrac{1}{2})$ \cite[p.~110]{Garnett}.

If the desired range is $\C \setminus (F \cup G)$, we let $\wt{F}$ and $\wt{G}$ be the 
images of $F$ and $G$ respectively under the multivalued map $z \mapsto z^{1/2}$ -- which are of capacity zero \cite[p.~137]{MR1334766}.  Let 
$\psi$ be the covering map from $\D$ onto 
$\C \setminus (\wt{F} \cup \wt{G} \cup [1,\infty))$, and let $\phi = \psi^2$.  The proof that mapping 
has the required properties is similar to the proofs above.
\end{proof}

\begin{Remark}
\begin{enumerate}
\item The mapping $\phi$ constructed above is actually outer, as long as $0$ is not in the range. 
To see this, observe that, in the first part of the proof, 
$\psi(\D) \subset \{\Re z > 0\}$. 
Such  functions are outer \cite[p.~109]{Garnett}. Since the product of outer functions is another outer function, 
this means that $\phi = \psi^4$ is outer. 
In the other cases in which $0$ is not in the range, we have that $\phi$ is a map onto a domain with a linear slit 
containing $0$ and going to $\infty$.   
But any map onto such a domain is outer, since we can define the square root of such a mapping, 
and that mapping will be onto a half plane omitting $0$ and thus outer (and belong to $H^p$ for all $p \in (0, 1)$ \cite[p.~109]{Garnett}).   
\item The proof of Theorem \ref{MT} also
shows that we can find a 
$\phi$ with the desired properties that is in $H^p$ for each $p \in (0,\tfrac{1}{4})$. 
\end{enumerate}
\end{Remark}

\section{Controlling the valence}

A key step in the construction in Theorem \ref{MT} was the uniformization theorem \cite[p.~125]{MR1344449}. However, it is not clear from our construction how one can control the valence of $\phi$. 

In this regard, one can ask the following question: Suppose we are given an closed set $E \subsetneq \R$ and a pair $(m, n)$, $m, n \in \N_0 \cup \{\infty\}$. Can we find a $\phi \in N^{+}_{\R}$ such that the valence of $\phi$ is equal to $m$ on $\C_{+}$,  $n$ on $\C_{-}$, and such that $\phi(\D) = \C \setminus E$? Can we say anything about the valence of $\phi$ on $\R \setminus E$?

Let us start with a few observations. Extending a standard proof of the open mapping theorem for analytic functions, one can prove the following. Recall that $v_{\phi}$ is the valence function from \eqref{bbbbbbx} and $d_{\phi}$ is the deficiency index from \eqref{668wyeuhfjw}. 

\begin{Proposition}
For $\phi \in N_{\R}^{+}$ and $N = 1, 2, \ldots$, the set 
$$\{w \in \C: v_{\phi}(w) \geq N\}$$ is an open subset of $\C$. 
\end{Proposition}
Note that the above result does not hold when $N = \infty$. 

\begin{Example}
Theorem \ref{MT} says that we can find a $\phi \in N_{\R}^{+}$ such that 
$$\phi(\D) = \C \setminus [0, \infty).$$ However, the previous proposition says that we can't find a $\phi \in N_{\R}^{+}$ with the same range and that also satisfies 
$$v_{\phi}|_{\C_{+}}  = 1, \quad v_{\phi}|_{\C_{-}} = 2, \quad v_{\phi}|_{\{x < 0\}} = 3$$
since $v_{\phi}|_{\C_{+}}$ and $v_{\phi}|_{\C{-}}$ must be at least $3$.
\end{Example}

Next we explore when the valence is finite. 

\begin{Proposition}\label{oriuioewur}
For $\phi \in N_{\R}^{+}$ the following are equivalent. 
\begin{enumerate}
\item[(i)] $\phi$ is a rational function; 
\item[(ii)] $v_{\phi}|_{\C_{+}}$ and $v_{\phi}|_{\C_{-}}$ are finite; 
\item[(iii)] There are two relatively prime finite Blaschke products $B_1, B_2$ such that $B_1 - B_2$ has no zeros on $\D$ and such that 
\begin{equation}\label{pppsd9sd97fd8}
\phi = i \frac{B_1 + B_2}{B_1 - B_2}.
\end{equation}
\item[(iv)] $d_{\phi}(i)$ and $d_{\phi}(-i)$ are finite. 
\end{enumerate}
Furthermore, 
\begin{enumerate}
\item[(a)] if any of the above equivalent conditions hold, we have 
$$v_{\phi}|_{\C_{+}} = \operatorname{deg}(B_2), \quad v_{\phi}|_{\C_{-}} = \operatorname{deg}(B_1).$$
\item[(b)] if any of the above conditions do not hold then either $v_{\phi}|_{\C_{+}}$ or $v_{\phi}|_{\C_{-}}$ is infinite nearly everywhere. 
\end{enumerate}
\end{Proposition}

\begin{proof}
$(i) \iff (iv)$ is from \cite{Helson}. 

$(iii) \implies (i)$: One can show directly that a $\phi$ given by \eqref{pppsd9sd97fd8} is a rational function in $N_{\R}^{+}$. 

$(i) \implies (iii)$: Set $g  = (\phi - i)/(\phi + i)$ and observe that $g$ is meromorphic function on $\D$ that is continuous with unimodular boundary values on $\T$. A classical theorem of Fatou \cite{GRM} says that 
$$g = \frac{B_1}{B_2},$$ where $B_1$ and $B_2$ are relatively prime Blaschke products. The result now follows with a simple computation. 

$(iii) \implies (ii)$: We follow an argument from \cite{GRM}. The M\"{o}bius transformation
$$\psi(z) = \frac{z - i}{z + i}$$ is injective and maps $\C_{+}$ onto $\D$ and $\C_{-}$ onto $\widehat{\C} \setminus \D^{-}$. If $w \in \C_{+}$, the number of solutions to $\phi(z) = w$ is the same as the number of solutions to $\psi \circ \phi (z)= \psi(w) = \eta \in \D$. Writing this out, this is same as the number of solutions to 
$$\eta = \frac{\phi(z) - i}{\phi(z) + i} = \frac{B_2(z)}{B_1(z)}.$$
To examine the number of zeros in $\D$ of $B_2 - \eta B_1$, observe that on $\T$ we have 
$$|\eta B_1| = |\eta| < 1 = |B_2|$$ and so, by Rouche's Theorem, the number of zeros in $\D$ of $B_2$ and $B_2 - \eta B_1$ are the same. This proves that $v_{\phi}(w) = \operatorname{deg}(B_2)$ whenever $w \in \C_{+}$.  The corresponding valence on $\C_{-}$ follows in a similar way. This also verifies $(a)$. 

$(ii) \implies (iv)$: Suppose $v_{\phi}|_{\C_{+}}$ is finite but $d_{\phi}(i) = \infty$. Since $d_{\phi}|_{\C_{+}}$ is constant, we see that $d_{\phi}(\lambda) = \infty$ for all $\lambda \in \C_{+}$. By \cite{MR0235151}, the inner factor $\Theta_{\lambda}$ of $\phi  - \lambda$ is a Blaschke product for all $\lambda \not \in \R$ except possibly for a set of capacity zero. As discussed in \eqref{bbbbbbx},
$$v_{\phi}(\lambda) = \operatorname{deg}(\Theta_{\lambda}) = \infty$$ for nearly all $\lambda \in \C_{+}$, a contradiction. An analogous argument holds for $v_{\phi}|_{\C_{-}}$. This also proves $(b)$. 
\end{proof}

The techniques in the above proof also show the following. 

\begin{Corollary}
If $\phi \in N_{\R}^{+}$ and $v_{\phi}|_{\C_{+}}$ is non-constant, then $v_{\phi}(\lambda) = \infty$ for nearly all $\lambda \in \C_{+}$. An analogous result holds for $v_{\phi}|_{\C_{-}}$
\end{Corollary}

\begin{Remark}\label{ooookkk}
We point out a paper \cite{MR554397} which gives some information about how one can, under mild technical conditions, define an inner function whose valence can be prescribed on various closed subsets of the $\D$ of zero capacity. Since real Smirnov functions take the form $i (u + v)/(u - v)$ (recall Helson's characterization from \eqref{sdhfjsd;gfgee2}), where $u$ and $v$ are inner, this creates examples of real Smirnov functions with wild valence behavior.
\end{Remark} 

We now characterize the possible valences of real Smirnov functions of finite valence. As seen in Proposition \ref{oriuioewur}, these are the rational real Smirnov functions. 

A {\itshape disk tree}
is a type of Riemann surface made by welding together copies 
of $\D$ and 
$$\D^* := \{z \in \widehat{\C}: |z| > 1 \text{ or } z = \infty\}.$$
To each copy of $\D$ and $\D^*$ in the disk tree is associated a positive 
integer $m$ which we call the valence.
A {\itshape admissible} 
arc on a disk tree is an arc on the boundary of a copy of 
$\D$ or $\D^*$ of the form  
$\{e^{i\theta} : a < \theta < b\}$ where $(a,b)$ contains no multiple of 
$2\pi / m$.  The {\itshape image arc} for a given admissible arc 
is the arc
$\{e^{i\theta} : am < \theta < bm\}.$  This is the image of the admissible arc 
under the function $z^m$.  
An admissible arc is called a {\itshape free arc} if it is part of the 
boundary of the disk tree, in other words, if it has 
{\itshape not}
been welded to 
another arc in the disk tree.  An admissible arc is a 
{\itshape welded arc} if it has been welded to another arc in the 
disk tree.

We will now formally define disk trees inductively.
A copy of $\D$ or $\D^*$ with a valence is a disk tree.
Let $X$ be a disk tree.  Let $Z$ be a new copy of $\D$ or $\D^*$ with valence 
$m$.  Let $Y$ be a copy of $\D$ or $\D^*$ in $X$ with valence $n$, where
$Y$ is a copy of $\D$ if $Z$ is a copy of $\D^*$, and vice versa. 
Suppose we are given a free arc on $Y$ and a free arc on $Z$  
and that both of them have the same image arc.  
We may weld $Y$ and $Z$ together on their 
free arcs by the map $Y \ni z \mapsto z^{n/m} \in Z$. 
See \cite[II.3C]{Ahlfors-Sario} for more on welding Riemann surfaces.
We will also give an explicit example of welding in 
Example \ref{ex:twodisctree}. 
We say the resulting surface of $X$ welded to $Z$ 
is a disk tree if it still has a free arc remaining. 

Any disk tree is simply connected (van Kampen's theorem \cite[Ch.~1]{MR1867354}).  It is also conformally equivalent to the 
disk.  Indeed, by the uniformization theorem, it is equivalent to the
disk, the plane, or the Riemann sphere.
But we may weld another copy of the disk (or complement of the disk)
to any disk tree and still 
obtain a simply connected Riemann surface, since any disk tree has 
a free arc.  Thus we can obtain the original 
disk tree from the new Riemann surface by removing a set with 
infinitely many points.  But if we take a set with infinitely many points 
away from either the Riemann sphere, the plane, or the disk, and we 
are left with a simply connected set, that set must be conformally 
equivalent to the disk. 

For a disk tree $X$, define 
$$f_X: X \mapsto \widehat{\C} \setminus \{1\}$$ by 
$f_X(z) = z^m$ for $z$ in a copy of $\D$ or $\D^*$ with valence $m$, or for 
$z$ in a welded arc that belongs to a copy of $\D$ with valence $m$. 
The function $f_X$ is clearly meromorphic in each copy of $\D$ or 
$\D^*$.  By construction, it is continuous in a neighborhood of each welded 
arc. To see this, suppose that some admissible arc $I$ in a copy of 
$\D$ or $\D^*$ with valence $n$ (call the copy $Y$) 
is welded to an admissible arc $J$ in a copy 
of $\D$ or $\D^*$ of valence $m$ (call the copy $Z$).  Let 
$Y'$ and $Z'$ be sufficiently small neighborhoods of $I$ and $J$ in 
$Y \cup I$ and $Z \cup J$, respectively.
Note that by definition of the welding the map 
\[
\phi(z) = 
\begin{cases} 
z^{n/m} &\text{ if $z \in Y'$}\\
z  &\text{ if $z \in Z'$}\\
\end{cases}
\]
is well defined, continuous and even conformal.  
But $f_X(z) = \phi(z)^m$ in $Y' \cup Z'$.  
Thus by Morera's theorem, $f_X$ is analytic in a neighborhood of each 
welded arc and thus meromorphic on $X$.
When restricted to a copy of 
$\D$ or $\D^*$ that has valence $m$ 
the function $f_X$ has valence $m$
at each point of
$\D$ or $\D^*$ respectively, and valence $0$ elsewhere.  
Also, when restricted to a welded
arc, the function $f_X$ has valence $1$ on the image arc
and $0$ elsewhere.
Thus, $f_X$ has valence on $\D$ equal to the sum of the valences of
the copies of $\D$ in the disk tree, and similarly for $\D^*$.
Its valence on a point in $\partial \D \setminus \{1\}$ is
equal to the number of image arcs in which it appears,
where if two arcs are welded together in the disk tree we
count their image arc (which is the same for both the welded arcs)
as appearing only once. 
Let 
$\phi : \D \mapsto X$ be a conformal map from $\D$ to $X$.  Then 
$f_X \circ \phi$ is a (meromorphic) 
map from $\D$ to $\widehat{\C} \setminus \{1\}$. 

\begin{Example}\label{ex:twodisctree}
  We give an explicit construction of a disk tree.
  This disk tree will have valence $1$ on $\D$ and
  $2$ on $\D^*$.  
See Figure \ref{figchart} for an illustration of some aspects of this 
example. 
Let 
$X$ consist of one copy each of $\D$ and $\D^*$, together with 
the boundary of $\partial \D \setminus \{1\}$.
For $0 < \theta < 2 \pi$, 
identify the point 
$e^{i \theta}$ on $\partial \D \setminus \{1\}$ with 
the point $e^{i \theta / 2}$ on $\partial \D^* \setminus \{1\}$. 
We will weld along these 
identified boundary points.  To do this explicitly, 
let \[ \overline{\D} \supset U_1 = \{re^{i\theta} : 1/2 < r \leq 1 \text{ and } 
0 < \theta < 2\pi\}\]
and
\[ \overline{\D^*} \supset U_2 = \{re^{i\theta} : 1 \leq r < 2 \text{ and } 
0 < \theta < \pi\}\]
and let $U = U_1 \cup U_2$.   
Take coordinate charts $\phi_1: \D \rightarrow \C$ and 
$\phi_2: \D^* \rightarrow \C$ and 
$\phi_3: U \rightarrow \C$, where 
$\phi_1(z) = z$ and $\phi_2(z) = 1/z$ and 
\[
\phi_3(z) = 
\begin{cases}
z^{1/2} &\text{if $z \in U_1,$}\\
z & \text{if $z \in U_2$}
\end{cases}
\]
where we take the branch of $z^{1/2}$ with $(-1)^{1/2} = i$ and branch 
cut along the positive real axis.
We can take as a basis for the open sets in $X$ sets that are open in 
$\D$ or in $\D^*$ or sets that are the inverse images of 
open sets under $\phi_3$.  Thus $U$ is an open set. 
(Note that any set that is the inverse image of an open set under 
$\phi_3$ and lies entirely in $\D$ is open in $\D$; the same may be 
said for $\D^*$.)
Notice that $\phi_3$ is continuous, and that
$\phi_1(\phi_3^{-1}(z)) = z^2$ and 
$\phi_2(\phi_3^{-1}(z)) = z^{-1}$.  Both of these maps are analytic in 
their domains.  Thus we have made $X$ into a Riemann surface with the 
given charts.  

Define \[
f_X(z) = 
\begin{cases}
z &\text{ for $z \in \D$ or $z \in \partial \D$}\\
z^2 &\text{ for $z \in \D^*$ or $z \in \partial \D^*$}.
\end{cases}
\]
Then $f_X$ is analytic from $X$ into $\widehat{\C} \setminus \{1\}$.  
To see this, note that 
$$f_X(\phi_1^{-1}(z)) = z, \quad f_X(\phi_2^{-1}(z)) = 1/z^2, \quad f_X(\phi_3^{-1}(z)) = z^2.$$
Also, $f$ has valence $1$ on $\D$, valence $2$ on $\D^*$, and
valence $1$ on $\partial \D \setminus \{1\}$.

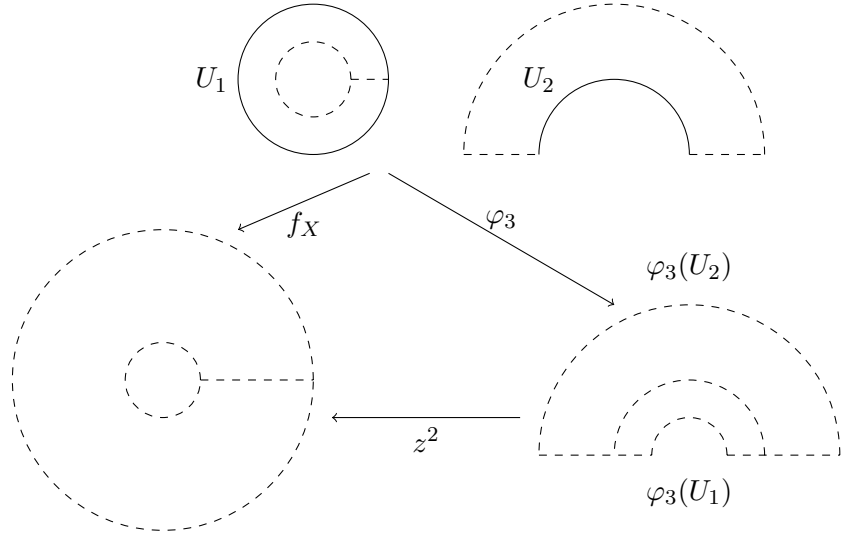
\begin{figure}
\begin{tikzpicture}
\draw[dashed] (0,1) circle (.5cm) (.5,1)--(1,1);
\draw (0,1) circle (1cm) (-1,1) node[left]{$U_1$};
\draw (5,0) arc (0:180:1);
\draw (3,1) node{$U_2$};
\draw[dashed] 
     (6,0) arc (0:180:2)
     (5,0)--(6,0)
     (2,0)--(3,0);
\draw[dashed] (6,-4) arc (0:180:1)
     (7,-4) arc (0:180:2)
     (5.5,-4) arc (0:180:.5)
     (3,-4)--(4.5,-4)
     (5.5,-4) -- (7,-4);
\draw[dashed] (-2,-3) circle (.5cm)
   (-2,-3) circle(2cm)
   (-1.5,-3)--(0,-3);
\draw[->] (.75,-.25)--(-1,-1) node[midway,right,below]{$f_X$};
\draw[->] (1,-.25)--(4,-2) node[midway,right,above]{$\phi_3$};
\draw[->] (2.75,-3.5) -- (.25, -3.5) node[midway,below]{$z^2$};
\draw (5,-1.5) node {$\phi_3(U_2)$};
\draw (5,-4.5) node {$\phi_3(U_1)$};
\end{tikzpicture}
\caption{\label{figchart} An illustration of $\phi_3$ and 
$f_X$ on $U$, from Example \ref{ex:twodisctree}.}
\end{figure}

\end{Example}

We now define some types of graphs which we need to state the main 
result. 
A {\itshape plane valence tree}
is a graph that is a tree. To each node is associated a
label of either $\C_+$ or $\C_-$ and 
positive integer $m$, called the valence.  A node with label
$\C_+$ may only be adjacent to nodes labeled $\C_-$, and nodes
labeled $\C_-$ may only be adjacent to nodes labeled
$\C_+$.  To each edge is associated 
an open interval in $\R$.  The interval may be all of
$\R$ but may not be empty.  We make the requirement that
a disjoint union of all the intervals on edges coming from a node is a 
subset of a disjoint union of $m$ copies of $\R$, 
where $m$ is the valence of the node.  
We require that some node has the property that 
a disjoint union of $m$ copies of $\R$, 
where $m$ is the valence of the node, 
contains a disjoint union of all the intervals on edges coming from the node, 
as well as another open interval in $\R$.  We say that such a node
has a free interval.  
For a plane valence tree, the valence of a point in $\C_+$ is the sum of 
the valences of the $\C_+$ nodes, and similarly for $\C_-$.  For a point 
in $\R$, it is the number of times it appears in an edge of the valence 
tree.

Define
$$\psi = -i \frac{z + 1}{z - 1}.$$ Then
$\psi$ maps $\D$ to $\C_+$ and $\D^*$ to $\C_-$ and
$\R$ to $\partial \D \setminus \{1\}$.
We may form a {\itshape disk valence tree}
by mapping $\C_+$ and $\C_-$ and $\R$ in the 
labeling of the valence tree 
to $\D$ and $\D^*$ and $\partial \D \setminus \{1\}$
under $\psi^{-1}$.

\begin{Example}
\begin{figure}[ht]
\begin{tikzpicture}[hp/.style={rectangle,draw},
sibling distance=10em
  ]
  \node[hp] {$\C_+$: 2}
    child { node[hp] {$\C_-:5$} 
      edge from parent
      node[left]{$(0,1)\ $}}
    child { node[hp] {$\C_- : 2$}
      child {node[hp] {$\C_+ : 1$}
        child {node[hp] {$\C_- : 1$}
          edge from parent
          node[right]{$\ (7,8)$}
        }
        child {node[hp] {$\C_- : 1$}
          edge from parent
          node[right]{$\ (9,10)$}
        }
        edge from parent
        node[right]{$(-3,5)$}
      }
      edge from parent 
      node[right]{$\ (-3,5)$}
    }
    ;
\end{tikzpicture}
\caption{\label{fig:expv}An example plane valence tree.}
\end{figure}
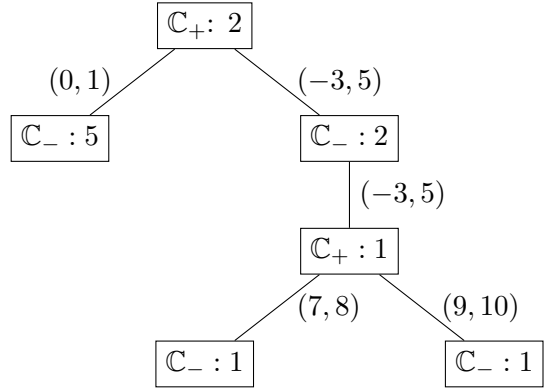
Figure \ref{fig:expv} is an example of a plane valence tree.
By Theorem \ref{thm:VT} (see below), 
there is a real Smirnov function with valence $3$ on 
$\C_+$, valence $9$ on $\C_-$, valence $3$ on $(0,1)$, valence 
$2$ on $(-3,0] \cup [1,5)$, valence $1$ on $(7,8)$ and $(9,10)$, 
and valence $0$ elsewhere. 

\end{Example}

\begin{Theorem}\label{thm:VT}
The valence of every real Smirnov function with finite valence is given 
by the valence of a plane valence tree, and any valence arising from a 
plane valence tree is the valence of a real Smirnov function.
\end{Theorem}

The proof of Theorem \ref{thm:VT} needs the following valence result. 
It is an exercise in \cite{Hayman} (Example 3.1 of Section 3.3), but 
for the sake of completeness we give the proof. 

\begin{Lemma}\label{poirt9fdg}
Let $f$ be an analytic function in $\D$ of valence at most $m$.
Then $f \in H^p$ for every $p \in (0, \tfrac{1}{2 m})$. 
\end{Lemma}

\begin{proof}
Recall the definition of the $p$-integral means $M_{p}(r, f)$ from \eqref{Mp} and define 
$$M_{\infty}(r, f) := \sup_{|z| = r} |f(z)|.$$
By \cite[Thm.~1]{Cartwright} (see also \cite[Sec.~2.3]{Hayman}) we have 
\begin{equation}\label{minty}
M_\infty(r,f) = O\left(\frac{1}{(1 - r)^{2 m}}\right).
\end{equation}
From \cite[Theorem 3.2]{Hayman} we see that if $f$ is $m$-valent and 
$0 < r_0 < r < 1$ then 
\[
M_p(r,f) \leqslant M_\infty(r_0,f)^p + m \max\Big(m,\frac{m^2}{2}\Big) 
  \int_{r_0}^r \frac{M_\infty(t,f)^p}{t} \, dt.
\]
Applying the estimate in \eqref{minty} for $M_\infty(r,f)$ shows that the function 
$$r \mapsto M_p(r,f)$$ is bounded when $p \in (0, \tfrac{1}{2 m})$, i.e., $f \in H^p$ for all $p \in (0, \frac{1}{2m})$.
\end{proof}

\begin{proof}[Proof of Theorem \ref{thm:VT}]
Given a plane valence tree, form the associated (disk) valence tree 
under the mapping $\psi^{-1}$. 
Construct a disk tree $X$ where nodes labeled $\D$ with valence $m$ correspond 
to disks of valence $m$, and similarly for $\D^*$.  
If $Y$ is a copy of $\D$ or $\D^*$ with valence $m$,
consider the arcs labeling edges
connected to the corresponding node.
Call these arcs $I_1, \ldots, I_k$.  
The disjoint union of these arcs
are a subset of a disjoint union of $m$ copies of
$\partial \D \setminus \{1\}$.
This means that we can find a disjoint set of admissible arcs on
the boundary of $Y$ whose image arcs are precisely the arcs
$I_1, \ldots, I_k$.

Given copies of $\D$ and $\D^*$ with corresponding nodes connected by an edge, 
weld them together
together on arcs with image arcs equal to the arc labeling the 
edge between them.  
This is possible by the above remarks.
Since the valence tree has the free arc property
and is a tree, $X$
will have a free arc on some copy of $\D$ or $\D^*$ contained
in it and will be a disk tree.  

Let $\phi$ be a conformal map from $\D$ onto $X$.  
The map $g = \psi \circ f_X \circ \phi$ has valence 
equal to the valence of the plane valence tree.  Since $g$ has finite 
valence, it is in the Smirnov class by Lemma \ref{poirt9fdg}. 
Every 
point in 
$\C \setminus (\R \cup \{i,-i\})$ has a neighborhood $U$ such that 
$g^{-1}(U)$ consists of disjoint sets that are each homeomorphic to $U$ 
under $f$.
The reason that we exclude $i$ and $-i$ is that $i=\psi(0)$ and
$-i=\psi(\infty)$, and the points $0$ and $\infty$ may be the image under
$f_X$ of points where $f_X$ has zero derivative. 
The fact that $\{i,-i\}$ has zero capacity together with
the same argument 
used to prove \eqref{oos9d9} shows that $g$ has real boundary values almost 
everywhere.
We mention in passing that, in fact, since $g$ has finite 
valence and thus is rational (by the 
result from \cite{Helson}), it
has real boundary 
values everywhere, except for a finite number of points where it has
$\infty$ as a boundary value.

For the other direction, suppose that $f$ is real Smirnov with finite valence. 
Then $f$ is rational and thus continuous on $\mathbb{D}$, when viewed as a map 
into $\widehat{\C}$. 

Consider $f^{-1}(-\infty, \infty)$.  This is a set of branched analytic 
arcs in $\D$, with endpoints only at $\D$ and branch points only at 
points where $f'$ is zero.  Each branch point has an even number of analytic 
arcs coming from it. 
Also, $f^{-1}(\C_+)$ is a finite disjoint union of open sets we will call 
upper regions (similarly for $\C_-$ and lower regions). 
By the maximum principle for harmonic functions 
applied to the real part of $f$,
no upper region can have more than two of 
the analytic arcs going into a branch point as boundary arcs. 
Also by the maximum principle, no two upper regions can share a boundary curve, 
and each upper region is simply connected.  
Given an upper region, note that its boundary curve is mapped into 
$\mathbb{R}$, and so by the argument principle each point in 
$\C_+$ has the same valence $m$ under $f$ restricted to the upper region.  
This means the boundary of the upper region must contain some point 
that maps to $\infty$.
The same applies to lower regions.
By the orientation preserving properties of analytic functions, each
upper region maps its boundary (considered as positively oriented) to 
the real line in a way so that moving along the boundary increases the 
value on the real line (except at $\infty$).  A lower region has 
the opposite property.  
Thus, by the argument principle, the boundary of an upper or lower region 
must map onto $\R \cup \{\infty\}$
exactly $m$ times. 
We give Figures \ref{fig:discregiona} and \ref{figb} 
for illustration. 

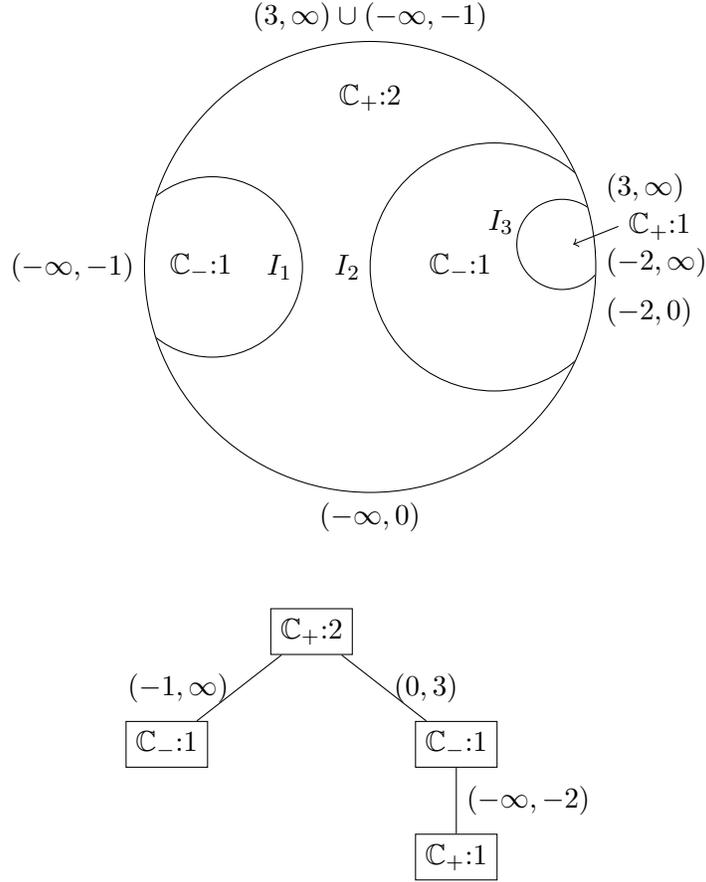
\begin{figure}
\begin{tikzpicture}[scale = 3.0]
\draw (.55,0) circle (.55cm);
\draw (-.7,0) circle (.4cm);
\draw (.85,.1) circle (.2cm);
\draw (0,.75) node{$\C_+$:2};
\draw (-.75,0) node{$\C_-$:1};
\draw (.4,0) node{$\C_-$:1};

\filldraw[color=white, even odd rule] (0,0) circle (1cm)
    (-1.5,-1.5) -- (-1.5,1) -- (1.5,1) -- (1.5,-1.5) -- cycle;
\draw (0,0) circle (1cm);
\draw[->] (1.1,.18) node[right]{$\C_+$:1} -- (.9,.1); 
\draw (-1,0) node[left]{$(-\infty,-1)$};
\draw (0,-1) node[below]{$(-\infty,0)$};
\draw (1,-.2) node[right]{$(-2,0)$};
\draw (1,.02) node[right]{$(-2,\infty)$};
\draw (1,.35) node[right]{$(3,\infty)$};
\draw (0,1) node[above]{$(3,\infty)\cup(-\infty,-1)$};
\draw (-.3,0) node[left]{$I_1$};
\draw (0,0) node[left]{$I_2$};
\draw (.68,.2) node[left]{$I_3$};
\end{tikzpicture}

\begin{tikzpicture}[hp/.style={rectangle,draw},
sibling distance=10em
  ]
  \node[hp]{$\C_+$:2}
  child { node[hp] {$\C_-$:1}
  edge from parent node[left]{$(-1,\infty)$}}
     child { node[hp] {$\C_-$:1} 
       child { node[hp] {$\C_+$:1} 
       edge from parent node[right]{$(-\infty,-2)$}
       }
       edge from parent node[right]{$(0,3)$}
     }
       ;
\end{tikzpicture}

\caption{\label{fig:discregiona}$I_1=(-1,\infty)$, $I_2=(0,3)$, $I_3=(-\infty,-2)$.
Valence is $3$ on $\C_+$, $2$ on $\C_-$, $2$ on $(0,3)$, $1$ on  
$(-1,0]$ and $[3,\infty)$, 
and $1$ on $(-\infty,-2)$. Note that as we proceed counterclockwise 
around the boundary of an upper region and 
clockwise around the boundary of a lower region, the 
image under $f$ increases on the real line.}
\end{figure}

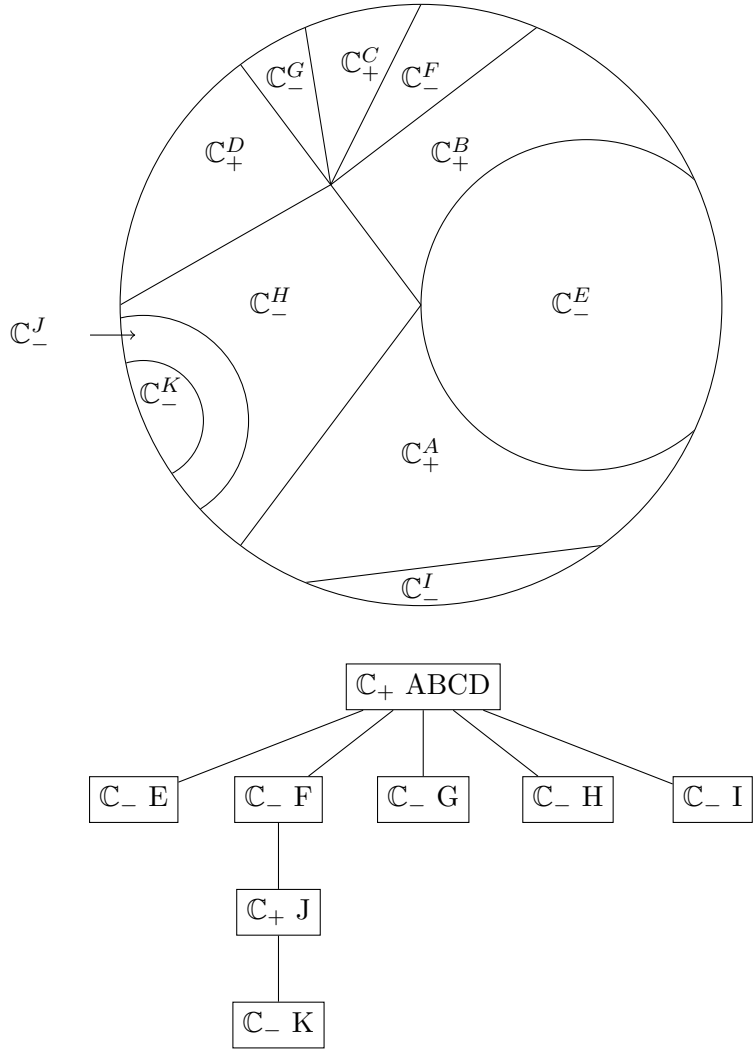
\begin{figure}
\begin{tikzpicture}[scale = 4.0]

\draw (.55,0) circle (.55cm);
\draw (0,0) -- (-.6,.8);
\draw (0,0) -- (-.6,-.8);
\draw (-.3,.4) -- (0,1);
\draw (-.3,.4) -- (5/13,12/13);
\draw (-.3,.4) -- (-1,0);
\draw (-.3,.4) -- (-5/13,12/13);
\draw (-12/13,-5/13) circle (.2cm);
\draw (-12/13,-5/13) circle (.35cm);
\draw (-5/13,-12/13) -- (3/5,-4/5);
\filldraw[color=white, even odd rule] (0,0) circle (1cm)
    (-1.5,-1.5) -- (-1.5,1) -- (1.5,1) -- (1.5,-1.5) -- cycle;
\draw (0,0) circle (1cm);

\draw (.5,0) node{\bfseries $\C_-^E$};
\draw (-.5,0) node{\bfseries $\C_-^H$};
\draw (0,-.5) node{\bfseries $\C_+^A$};
\draw (.1,.5) node{\bfseries $\C_+^B$};
\draw (-.2,.8) node{\bfseries $\C_+^C$};
\draw (0,.75) node{\bfseries $\C_-^F$};
\draw (-.45,.75) node{\bfseries $\C_-^G$};
\draw (-.65,.5) node{\bfseries $\C_+^D$};
\draw (0,-.95) node {\bfseries $\C_-^I$};
\draw (-.865,-.3) node {\bfseries $\C_-^K$};
\draw (-1.3,-.1) node {\bfseries $\C_-^J$};
\draw[->](-1.1,-.1)--(-.95,-.1);

\end{tikzpicture}

\vspace{-.5in}
\begin{tikzpicture}[hp/.style={rectangle,draw},
sibling distance=5em
  ]
  \node[hp]{$\C_+$ ABCD}
    child { node[hp] {$\C_-$ E } }
    child { node[hp] {$\C_-$ F } 
      child { node[hp] {$\C_+$ J } 
        child { node[hp] {$\C_-$ K  }}} 
      }
    child { node[hp] {$\C_-$ G } }
    child { node[hp] {$\C_-$ H } }
    child { node[hp] {$\C_-$ I } }
    ;
\end{tikzpicture}

\caption{\label{figb} The disk with upper and lower regions, and the 
corresponding valence tree.  Upper regions are labeled $\C_+$ and lower 
regions are labeled $\C_-$. Edges and valences are not labeled.  
Note that, if 
we start with upper region A, we obtain the upper collection with 
A, B, C, and D as the first node.}
\end{figure}

We will now give a method that, given a finite valence real Smirnov function, 
constructs a plane valence tree such that the valence of the tree 
corresponds to the valence of the function.  
We proceed by induction on the number of regions.  
If there is only one region, and this region has valence $m$, construct 
a plane valence tree with only one node of valence $m$. 

Suppose there is more than one region.
We may replace upper regions by lower regions in the following argument if 
needed.  
Take an upper region, and consider all 
upper regions sharing common boundary points 
{\itshape inside the unit disk} with the given region.  
Take the union, including the boundary points and boundary arcs.  
Repeat the process for all the new upper regions added, and continue 
until it is no longer possible to do so.  This forms a finite union of 
upper regions (and their boundaries) -
call it $X$, and call it an upper collection.  
By the maximum principle for harmonic functions, $X$ is simply 
connected.  
$X$ is not the whole disk by our assumption. 

The 
boundary of $X$ must intersect the boundary of the disk since some point in 
the boundary of each upper region maps to $\infty$.  
Thus, the complement of $X$ consists of a union of simply connected 
sets.
For the $j^{th}$ set, let $\phi_j$ be the conformal 
map from the unit disk onto this set. By the induction hypothesis, 
we may form a valence tree for the $j^{th}$ set, using the function 
$f \circ \phi_j$ instead of the function $f$, and by starting with 
a lower region instead of an upper region.
(We could also apply the above reasoning directly to the $j^{th}$
set without using the conformal map $\phi_j$.)
Let $T_j$ denote the valence tree for the $j^{th}$ set.  
Now suppose the total valence of 
all components of $X$ is $M$.  Draw a $\C_+$ node with valence $M$.  
  Draw edges from the 
node for $X$ to the nodes in the trees $T_j$ 
that correspond to lower collections sharing boundary arcs with $X$.  
There is at least one edge to each $T_j$ because every $T_j$ shares a 
boundary arc with $X$. 
We will later see that there is exactly one edge to each $T_j$. 
Label each arc's edges with the intervals corresponding to the values of 
$f$ on the edges.  

We will show the graph formed is a tree.  
Consider an analytic arc in $f(-\infty,\infty)$ that approaches the 
boundary of the circle and is part of the boundary of the upper collection.  
If we start from a point in the arc that is on the unit
circle and follow the arc, 
it either terminates at another point of the circle, or at a branch 
point.  If this is the case, some other arc going from the branch point 
must be the boundary of the same lower region as the original arc; 
follow the new arc.  We may continue until we hit the boundary of the 
unit circle, which we must since the arc can never approach the same branch 
point again, and there are finitely many branch points.
Call the combination of arcs $\gamma$.  
The combination of arcs $\gamma$  
will be part of the boundary of some lower region, call it $L$. 
The  lower region $L$ can only have a common boundary with the 
upper collection $X$ 
along $\gamma$, since the upper collection $X$ is connected 
and $\gamma$ intersects the circle at its two ends.   
The component of the complement of the upper collection $X$ that 
contains $L$ can have common boundary with $X$ only on 
$\gamma$, for the same reason.   
This shows that the graph we form is a tree.  

The arc $\gamma$ (not counting points on the unit circle) maps to 
some subset of 
$(-\infty,\infty)$ once.  This follows since as a point travels along 
$\gamma$, the image of the point always increases on the real line, 
or always decreases, since $\gamma$ is part of the boundary of a 
lower region.  
This shows each edge is labeled with an interval in $\mathbb{R}$. 
Since the boundary of the upper collection maps onto
$\R \cup \{\infty\}$ exactly $M$ times, the disjoint
union of all the intervals labeling edges coming from the
upper collection is contained in
a disjoint union of $M$ copies of $\R$.

Note that 
there is some node
(say of valence $m$) such that m disjoint copies of $\R$ minus 
the disjoint union of its edges contains an (open) interval. 
This follows from the 
fact that some upper or lower region must have a boundary arc in 
common with the unit disk.
Thus, the graph we construct is a plane 
valence tree.
The valence of the 
tree is the same as that of $f$ by construction. 
\end{proof}

Note that if a holomorphic function has bounded finite valence on $\C_+$ and
$\C_-$, it must have finite valence on $\R$ by the open mapping theorem.

We now give some examples.

\begin{Example}\label{sdfjfdg}
Let $n \geqslant 1$ and 
$$(a_1, b_1), (a_2, b_2), \ldots, (a_n, b_n)$$ be a finite set of open intervals such that none of the intervals is the entire real line and 
$(a_j,b_j)$ is disjoint from $(a_{j+1}, b_{j+1})$ for each $j$.  Then there 
is a function from $N_{\R}^{+}$ whose range is 
\[
\bigcup_{j=1}^n (a_j, b_j) \cup \C_+ \cup \C_-.
\]
Moreover, the valence of each point of $\C_+$ is $\lfloor n/2 \rfloor + 1$ and 
the valence of each point of $\C_-$ is $\lceil n/2 \rceil$.  The valence 
of each point in $\R$ is equal to the number of the intervals 
$(a_j, b_j)$ in which it lies.
\end{Example}

Clearly we can interchange the roles of $\C_+$ and $\C_-$ in the above 
example.

We could deduce this from our previous theorem, but we will first give 
an independent proof that is simpler than the proof of the previous 
theorem. 
\begin{proof}%
Construct a Riemann surface as follows.  Weld a copy of 
$\C_+$ to $\C_-$ along the interval $(a_1, b_1)$.  
Now weld the copy of $\C_-$ to a different copy of 
$\C_+$ along the interval $(a_2, b_2)$. 
Now weld this copy of $\C_+$ to a different copy of 
$\C_-$ along the interval $(a_3, b_3)$.  
 Proceed in this manner 
until all of the intervals are exhausted. 
Call this Riemann surface $X$. 
Let $\theta$ be the 
projection map from $X$ to $\C$ that takes a given point 
in the Riemann surface to the corresponding point in either 
$\C_+$, $\C_-$, or $\R$.

We now claim that the Riemann surface is conformally equivalent to $\D$.  If it were not, then, since it is simply connected, it would be 
equivalent to either the Riemann sphere or the complex plane (uniformization theorem).  It is not 
equivalent to the sphere since it is not compact.  %
It is not equivalent to the complex plane since we could weld another 
half plane onto the last half plane welded onto
$X$ and still have a simply connected surface, and if 
we remove infinitely many points from any simply connected Riemann 
surface and are left with a simply connected surface, the new surface 
must be equivalent to the disk by the uniformization theorem and
the Riemann mapping theorem. 

Let $\phi$ be the conformal map 
from $\D$ to the Riemann surface and 
$$f =  \theta \circ \phi.$$ Then 
$f$ maps from $\D$ into $\C$.  Since $f$ has valence at most 
$\lfloor n/2 \rfloor + 1$ at each point, it belongs to some Hardy space (Lemma \ref{poirt9fdg}) and thus 
belongs to $N^+$.  
We will now show that 
$f$ has real (radial) boundary values almost everywhere. 
To see this, note that every point in 
$\C \setminus \R$ has a neighborhood $U$ such that 
$\theta^{-1}(U)$ consists of disjoint sets that are each homeomorphic to $U$ 
under $\theta$.  
Thus 
every point in 
$\C \setminus \R$ has a neighborhood $U$ such that 
$f^{-1}(U)$ consists of disjoint sets that are each homeomorphic to $U$ 
under $f$. 
The same argument as used to prove \eqref{oos9d9} shows that no (radial) boundary values of 
$f$ lie in $\C \setminus \R$.  
\end{proof}

For example, for $n=3$, 
the valence tree for this example is as shown below.
The nodes are shown with their valences. 

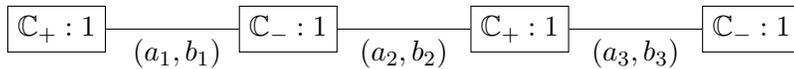
\begin{tikzpicture}[hp/.style={rectangle,draw},
grow = right, level distance=8em]
  \node[hp] {$\C_+ :1$}
    child { node[hp] {$\C_- : 1$}
      child {node[hp] {$\C_+ : 1$}
        child{node[hp] {$\C_- : 1$}
          edge from parent
          node[below]{$(a_3,b_3)$}
      }
        edge from parent
        node[below]{$(a_2,b_2)$}
      }
      edge from parent 
      node[below]{$\ (a_1,b_1)$}
    }
    ;
  \end{tikzpicture}

  \begin{Example}
    The only possible valence trees for a real Smirnov function with
    valence $1$ on $\C_+$ and valence $1$ on $\C_-$ have a
    $\C_+$ node of valence $1$ connected to a $\C_-$ node of
    valence $1$.  The interval labeling the edge connecting them
    can be any nonempty open interval, except $\R$, since if the
    edge was labeled $\R$ neither node would have a free interval.
    Thus the valence of such a real Smirnov function is $1$ on a
    nonempty open proper subinterval of $\R$, and $0$ elsewhere on
    $\R$.
    \end{Example}

  \begin{Example}
    Let us find all possible valences for real Smirnov functions with
    valence $2$ on $\C_+$ and valence $1$ on $\C_-$.
    The valence tree
  for such a function has either one $\C_+$ node of valence
  $2$, or two $\C_+$ nodes of valence $1$.  These nodes cannot be
  adjacent. Figure \ref{fig:vt21} shows all possible valence trees.
  In tree I, $I_1$ and $I_2$ must be disjoint; both the top and
  bottom node automatically have free intervals.
  In tree II, $I_1$ can be an arbitrary non-empty open interval.
  Because the $\C_+$ node has valence $2$ and only one edge, it
  has a free arc automatically.  
  So the valence at every point in $\R$ is either $1$ or $0$.  The range
  on $\R$ 
  is either the union of two open intervals, or is one interval.  
  
    \begin{figure}
         \begin{tikzpicture}[hp/.style={rectangle,draw},
sibling distance=3em
]
  \draw (-1,0) node{I};
  \node[hp]{$\C_+$:1}
  child { node[hp] {$\C_-$:1} 
       child { node[hp] {$\C_+$:1} 
       edge from parent node[left]{$I_2$}
       }
       edge from parent node[left]{$I_1\,$}
     }
     ;
  \draw (2,0) node{II};
  \draw(3,0) node[hp]{$\C_+$:2}
  child { node[hp] {$\C_-$:1} 
       edge from parent node[left]{$I_1\,$}
     }
     ;
   \end{tikzpicture}
   \caption{\label{fig:vt21} All possible valence tress with valence
     $2$ on $\C_+$ and $1$ on $\C_-$.}
\end{figure}
\end{Example}

\begin{Example}
  Let us find all possible valences for real Smirnov functions with
  valence $2$ on $\C_+$ and valence $3$ on $\C_-$.  The valence tree
  for such a function has either one $\C_+$ node of valence
  $2$, or two $\C_+$ nodes of valence $1$.  These nodes cannot be
  adjacent. Figure \ref{fig:vt23} shows all possible valence trees.
  Note that if all the other conditions for being a real valence tree
  are satisfied, some node in a tree must have a free interval if
  there is some node of valence one connected to at least two nodes, one of
  which connects to no other nodes.  
  If a node of valence $m$ has less than $m$ edges connected to it,
  the free interval condition is also automatically satisfied.
  These remarks apply to all the trees in Figure \ref{fig:vt23} except for
  tree VIII, which also automatically has a free interval since two of the
  intervals labeling its edges must be disjoint.  
  
  In tree I, we require $I_1$, $I_3$, and $I_4$ to be pairwise disjoint,
  and $I_1$ and $I_2$ to be disjoint.  In tree II, we require $I_1$ and
  $I_3$ to be disjoint.  The intervals $I_1$ and $I_2$ do not have to
  be disjoint, since the node they have in common has valence $2$, and
  any two intervals can fit disjointly into two copies of $\R$. 
  In tree III, we require $I_1$ and $I_3$ to be disjoint, and $I_1$ and
  $I_2$ to be disjoint.  In tree IV, the intervals $I_1$ and $I_2$ can
  be arbitrary.

  In tree V, intervals $I_1$ and $I_2$ must be disjoint, intervals
  $I_1$ and $I_3$ are disjoint, and intervals $I_3$ and $I_4$ are
  disjoint.  In tree VI, $I_1$ must be disjoint from
  $I_2$ and $I_2$ must be disjoint from $I_3$.
  In tree VII, $I_2$ must be disjoint from $I_3$.  
  
  In tree VIII, the intervals $I_1$, $I_2$, and $I_3$ must be able to fit into
  two disjoint copies of $\R$ without intersecting.  This is
  equivalent to requiring that two of them be disjoint.  In trees
  IX and X, there are no requirements on the intervals.  Note that
  for all cases, the conditions above imply that
  some node must have a free interval, and the above intervals are
  allowed to be $\R$ if this does not conflict with any of the above
  conditions.  

Considering all possible cases shows that
   the range of the function on $\R$ may be (counting
  multiplicity) %
  the union of four open intervals
  $I_1$, $I_2$, $I_3$ and $I_4$ where $I_1$ and $I_2$ are disjoint,
  $I_1$ and $I_3$ are disjoint, and $I_3$ and $I_4$ are disjoint.
  The range of the function on $\R$ (counting multiplicity) may also
  be the union of three open intervals, at least two of which
  are disjoint, or it may be the union of two open intervals; or it will
  be one open interval. These are the only cases. 
  Some of these intervals may be $\R$ if the conditions are satisfied
  (although they cannot be empty).  
  To take a concrete example, from the first case we see that the 
  range could be (counting multiplicities)
  $$(0,1) \cup (2, \infty) \cup (2, 3) \cup (4,5).$$  In other words,
  the valence would be one on $(0,1)$, two on $(2,3)$ and $(4,5)$,
  one on $[3,4] \cup [5,\infty)$, and zero on the rest
  of $\R$.

  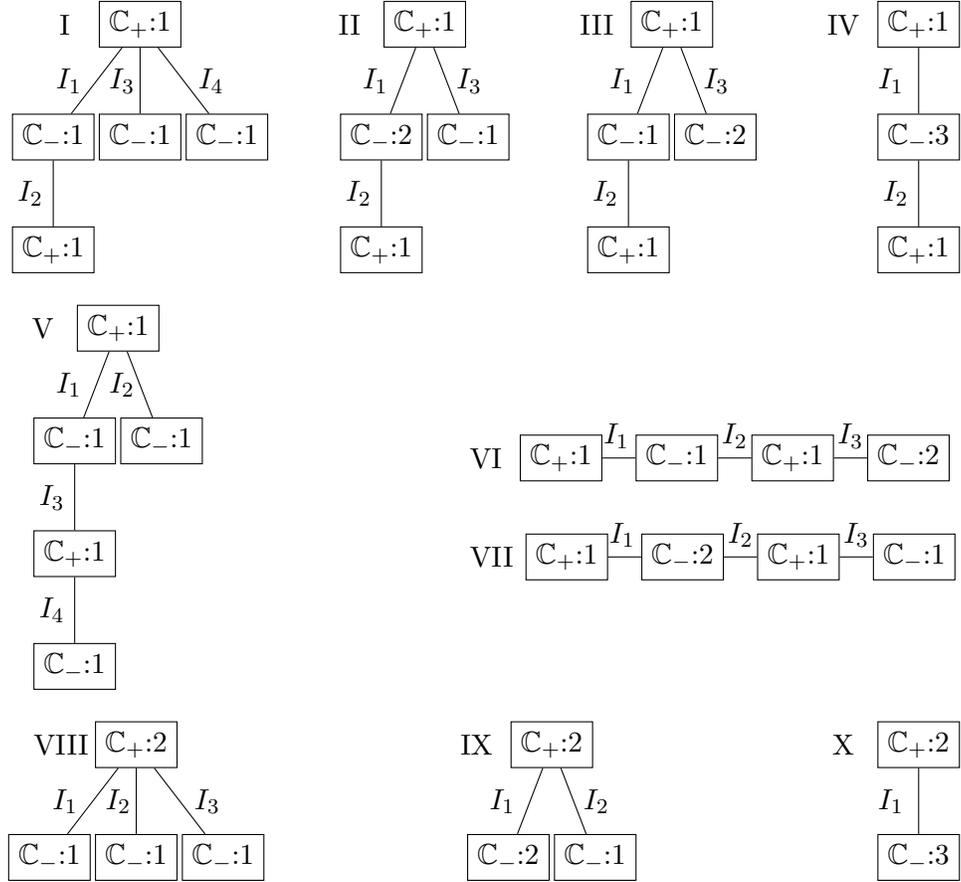
\begin{figure}[htbp!]
      \begin{tikzpicture}[hp/.style={rectangle,draw},
sibling distance=3em
]
  \draw (-1,0) node{I};
  \node[hp]{$\C_+$:1}
  child { node[hp] {$\C_-$:1} 
       child { node[hp] {$\C_+$:1} 
       edge from parent node[left]{$I_2$}
       }
       edge from parent node[left]{$I_1\,$}
     }
  child { node[hp] {$\C_-$:1}
    edge from parent node[left]{$I_3\!$}}
  child { node[hp] {$\C_-$:1}
    edge from parent node[right]{$\,I_4$}}
     ;
   \end{tikzpicture} \hfill
      \begin{tikzpicture}[hp/.style={rectangle,draw},
sibling distance=3em
  ]
  \draw (-1,0) node{II};
  \node[hp]{$\C_+$:1}
  child { node[hp] {$\C_-$:2} 
       child { node[hp] {$\C_+$:1} 
       edge from parent node[left]{$I_2$}
       }
       edge from parent node[left]{$I_1\,$}
     }
  child { node[hp] {$\C_-$:1}
    edge from parent node[right]{$I_3\!$}}
     ;
   \end{tikzpicture} \hfill
         \begin{tikzpicture}[hp/.style={rectangle,draw},
sibling distance=3em
  ]
  \draw (-1,0) node{III};
  \node[hp]{$\C_+$:1}
  child { node[hp] {$\C_-$:1} 
       child { node[hp] {$\C_+$:1} 
       edge from parent node[left]{$I_2$}
       }
       edge from parent node[left]{$I_1\,$}
     }
  child { node[hp] {$\C_-$:2}
    edge from parent node[right]{$I_3\!$}}
     ;
\end{tikzpicture}
\hfill
         \begin{tikzpicture}[hp/.style={rectangle,draw},
sibling distance=3em
]
  \draw (-1,0) node{IV};
  \node[hp]{$\C_+$:1}
  child { node[hp] {$\C_-$:3} 
       child { node[hp] {$\C_+$:1} 
       edge from parent node[left]{$I_2$}
       }
       edge from parent node[left]{$I_1\,$}
     }
     ;
   \end{tikzpicture}

   \vspace{1em}
   \begin{minipage}{.45\linewidth}
         \begin{tikzpicture}[hp/.style={rectangle,draw},
sibling distance=3em
]
  \draw (-1,0) node{V};
  \node[hp]{$\C_+$:1}
  child { node[hp] {$\C_-$:1} 
    child { node[hp] {$\C_+$:1}
      child{ node[hp] {$\C_-$:1}
        edge from parent node[left]{$I_4$}}
       edge from parent node[left]{$I_3$}
       }
       edge from parent node[left]{$I_1\,$}
     }
  child { node[hp] {$\C_-$:1}
    edge from parent node[left]{$I_2\!$}}
     ;
   \end{tikzpicture} \hfill
   \end{minipage}
\begin{minipage}{.5\linewidth}
   \begin{tikzpicture}[hp/.style={rectangle,draw},grow=right,
     level distance = 4em, 
sibling distance=5em
  ]
  \draw (-1,0) node{VI};
  \node[hp]{$\C_+$:1}
  child { node[hp] {$\C_-$:1} 
    child { node[hp] {$\C_+$:1}
      child { node[hp] {$\C_-$:2}
        edge from parent node[above]{$I_3$}
      }
       edge from parent node[above]{$I_2$}
       }
       edge from parent node[above]{$I_1\,$}
     }
     ;
   \end{tikzpicture}

   \vspace{1em}
   \begin{tikzpicture}[hp/.style={rectangle,draw},grow=right,
     level distance = 4em, 
sibling distance=5em
  ]
  \draw (-1,0) node{VII};
  \node[hp]{$\C_+$:1}
  child { node[hp] {$\C_-$:2} 
    child { node[hp] {$\C_+$:1}
      child { node[hp] {$\C_-$:1}
        edge from parent node[above]{$I_3$}
      }
       edge from parent node[above]{$I_2$}
       }
       edge from parent node[above]{$I_1\,$}
     }
     ;
   \end{tikzpicture}
   \end{minipage}
\hfill

\vspace{1em}
   \begin{tikzpicture}[hp/.style={rectangle,draw},
sibling distance=3em
  ]
  \draw (-1,0) node{VIII};
  \node[hp]{$\C_+$:2}
  child { node[hp] {$\C_-$:1} 
       edge from parent node[left]{$I_1\,$}
     }
  child { node[hp] {$\C_-$:1}
    edge from parent node[left]{$I_2\!$}}
  child { node[hp] {$\C_-$:1}
    edge from parent node[right]{$\,I_3$}}
     ;
   \end{tikzpicture} \hfill
      \begin{tikzpicture}[hp/.style={rectangle,draw},
sibling distance=3em
]
  \draw (-1,0) node{IX};
  \node[hp]{$\C_+$:2}
  child { node[hp] {$\C_-$:2} 
       edge from parent node[left]{$I_1\,$}
     }
  child { node[hp] {$\C_-$:1}
    edge from parent node[right]{$I_2\!$}}
     ;
   \end{tikzpicture} \hfill
         \begin{tikzpicture}[hp/.style={rectangle,draw},
sibling distance=3em
  ]
  \draw (-1,0) node{X};
  \node[hp]{$\C_+$:2}
  child { node[hp] {$\C_-$:3} 
       edge from parent node[left]{$I_1\,$}
     }
     ;
   \end{tikzpicture}
   \caption{\label{fig:vt23} All possible valence tress with valence
     $2$ on $\C_+$ and $3$ on $\C_-$.}
\end{figure}
\end{Example}

\def\cprime{$'$} \def\cprime{$'$}


\begin{thebibliography}{10}

\bibitem{Ahlfors-Sario}
Lars~V. Ahlfors and Leo Sario.
\newblock {\em Riemann surfaces}.
\newblock Princeton Mathematical Series, No. 26. Princeton University Press,
  Princeton, N.J., 1960.

\bibitem{A-G}
N.~I. Akhiezer and I.~M. Glazman.
\newblock {\em Theory of linear operators in {H}ilbert space. {V}ol. {II}}.
\newblock Translated from the Russian by Merlynd Nestell. Frederick Ungar
  Publishing Co., New York, 1963.

\bibitem{MR3004956}
Alexandru Aleman, R.~T.~W Martin, and William~T. Ross.
\newblock On a theorem of {L}ivsic.
\newblock {\em J. Funct. Anal.}, 264(4):999--1048, 2013.

\bibitem{Cartwright}
M.~L. Cartwright.
\newblock Some inequalities in the theory of functions.
\newblock {\em Math. Ann.}, 111(1):98--118, 1935.

\bibitem{MR1761913}
Joseph~A. Cima and William~T. Ross.
\newblock {\em The backward shift on the {H}ardy space}, volume~79 of {\em
  Mathematical Surveys and Monographs}.
\newblock American Mathematical Society, Providence, RI, 2000.

\bibitem{C-L}
E.~F. Collingwood and A.~J. Lohwater.
\newblock {\em The theory of cluster sets}.
\newblock Cambridge University Press, Cambridge, 1966.

\bibitem{MR779463}
Peter Colwell.
\newblock {\em Blaschke products}.
\newblock University of Michigan Press, Ann Arbor, MI, 1985.
\newblock Bounded analytic functions.

\bibitem{MR1344449}
John~B. Conway.
\newblock {\em Functions of one complex variable. {II}}, volume 159 of {\em
  Graduate Texts in Mathematics}.
\newblock Springer-Verlag, New York, 1995.

\bibitem{MR0270196}
R.~G. Douglas, H.~S. Shapiro, and A.~L. Shields.
\newblock Cyclic vectors and invariant subspaces for the backward shift
  operator.
\newblock {\em Ann. Inst. Fourier (Grenoble)}, 20(fasc. 1):37--76, 1970.

\bibitem{Duren}
P.~L. Duren.
\newblock {\em Theory of ${H}\sp{p}$ spaces}.
\newblock Academic Press, New York, 1970.

\bibitem{Fisher}
Stephen~D. Fisher.
\newblock {\em Function theory on planar domains}.
\newblock Pure and Applied Mathematics (New York). John Wiley \& Sons, Inc.,
  New York, 1983.
\newblock A second course in complex analysis, A Wiley-Interscience
  Publication.

\bibitem{GRM}
S.~Garcia, J.~Mashreghi, and W.~Ross.
\newblock {\em Finite Blaschke products and their connections}.
\newblock Monographs in Mathematics. Springer-Verlag, New York, 2019.

\bibitem{MR3526203}
Stephan~Ramon Garcia, Javad Mashreghi, and William~T. Ross.
\newblock {\em Introduction to model spaces and their operators}, volume 148 of
  {\em Cambridge Studies in Advanced Mathematics}.
\newblock Cambridge University Press, Cambridge, 2016.

\bibitem{GMR}
Stephan~Ramon Garcia, Javad Mashreghi, and William~T. Ross.
\newblock Real complex functions.
\newblock In {\em Recent progress on operator theory and approximation in
  spaces of analytic functions}, volume 679 of {\em Contemp. Math.}, pages
  91--128. Amer. Math. Soc., Providence, RI, 2016.

\bibitem{MR2021044}
Stephan~Ramon Garcia and Donald Sarason.
\newblock Real outer functions.
\newblock {\em Indiana Univ. Math. J.}, 52(6):1397--1412, 2003.

\bibitem{Garnett}
John~B. Garnett.
\newblock {\em Bounded analytic functions}, volume 236 of {\em Graduate Texts
  in Mathematics}.
\newblock Springer, New York, first edition, 2007.

\bibitem{MR1867354}
Allen Hatcher.
\newblock {\em Algebraic topology}.
\newblock Cambridge University Press, Cambridge, 2002.

\bibitem{Hayman}
W.~K. Hayman.
\newblock {\em Multivalent functions}, volume 110 of {\em Cambridge Tracts in
  Mathematics}.
\newblock Cambridge University Press, Cambridge, second edition, 1994.

\bibitem{Helson}
H.~Helson.
\newblock Large analytic functions.
\newblock In {\em Linear operators in function spaces ({T}imi\c soara, 1988)},
  volume~43 of {\em Oper. Theory Adv. Appl.}, pages 209--216. Birkh\"auser,
  Basel, 1990.

\bibitem{Helson2}
H.~Helson.
\newblock Large analytic functions. {II}.
\newblock In {\em Analysis and partial differential equations}, volume 122 of
  {\em Lecture Notes in Pure and Appl. Math.}, pages 217--220. Dekker, New
  York, 1990.

\bibitem{MR2986324}
Javad Mashreghi.
\newblock {\em Derivatives of inner functions}, volume~31 of {\em Fields
  Institute Monographs}.
\newblock Springer, New York; Fields Institute for Research in Mathematical
  Sciences, Toronto, ON, 2013.

\bibitem{MR1889082}
Alexei~G. Poltoratski.
\newblock Properties of exposed points in the unit ball of {$H^1$}.
\newblock {\em Indiana Univ. Math. J.}, 50(4):1789--1806, 2001.

\bibitem{MR1334766}
Thomas Ransford.
\newblock {\em Potential theory in the complex plane}, volume~28 of {\em London
  Mathematical Society Student Texts}.
\newblock Cambridge University Press, Cambridge, 1995.

\bibitem{MR1895624}
William~T. Ross and Harold~S. Shapiro.
\newblock {\em Generalized analytic continuation}, volume~25 of {\em University
  Lecture Series}.
\newblock American Mathematical Society, Providence, RI, 2002.

\bibitem{MR0235151}
Walter Rudin.
\newblock A generalization of a theorem of {F}rostman.
\newblock {\em Math. Scand}, 21:136--143 (1968), 1967.

\bibitem{MR2418122}
Donald Sarason.
\newblock Unbounded {T}oeplitz operators.
\newblock {\em Integral Equations Operator Theory}, 61(2):281--298, 2008.

\bibitem{MR554397}
Kenneth Stephenson.
\newblock The finite range of an inner function.
\newblock {\em Bull. London Math. Soc.}, 11(3):300--303, 1979.

\end{thebibliography}
\end{document}